\documentclass[a4paper, reqno]{amsart}
\usepackage{graphicx} 
\usepackage{amssymb}
\usepackage{amsmath}
\usepackage{enumerate}
\usepackage{enumitem}
\usepackage{amsthm}
\usepackage{mathtools}
\usepackage{todonotes}
\usepackage[backend=biber, sorting=nyvt, style=alphabetic, maxcitenames=4, isbn=false, maxbibnames=5, url=false]{biblatex}
\bibliography{franzi.bib}
\usepackage[french,english]{babel}

\newcommand{\reqnomode}{\tagsleft@false\let\veqno\@@eqno}
\newcommand{\Lring}{\mathcal{L}_\mathrm{ring}}
\newcommand{\Lval}{\mathcal{L}_\mathrm{val}}
\newcommand{\Oo}{\mathcal{O}}
\newcommand{\Loag}{\mathcal{L}_\mathrm{oag}}

\newtheorem{Theorem}[subsection]{Theorem}

\newtheorem*{thm*}{Theorem}
\newtheorem{thm}[subsection]{Theorem}
\newtheorem{cor}[subsection]{Corollary}

\newtheorem{Q}[subsection]{Question}
\newtheorem{lemma}[subsection]{Lemma}
\newtheorem{fact}[subsection]{Fact}

\newtheorem{prop}[subsection]{Proposition}
\newtheorem{ex}[subsection]{Example}
\newtheorem{remark}[subsection]{Remark}
\newtheorem{Def}[subsection]{Definition}
\theoremstyle{remark}
\newtheorem{claim}{Claim}
\newtheorem*{claim*}{Claim}
\newtheorem*{poc}{Proof of Claim}

\newenvironment{claimproof}{%
	\begin{poc}
	}{%
	\hfill$\Box_\textrm{Claim}$
	\end{poc}
	}

\newtheorem*{pen}{Proof}

\makeatletter
\newenvironment{abstracts}{%
	\footnotesize
  \newcommand{\abstractin}[1]{%
	\setlength{\parindent}{12pt}
    \selectlanguage{##1}%
    \item[\hskip\labelsep\scshape\abstractname.]%
  }%
  \global\setbox\abstractbox=\vbox \bgroup
    \hsize\textwidth \advance\hsize-8pc
    \list{}{\labelwidth\z@ \leftmargin3pc \rightmargin\leftmargin}%
}{%
  \endlist\egroup
}
\makeatother

\title{Perfectoid fields in the language of rings}
\author{Franziska Jahnke}
\address{Institute for Logic, Language and Computation, University of Amsterdam, 
Science Park 107,
1098 XG Amsterdam, The Netherlands and\newline
\indent Institute for Mathematical Logic and Foundations, Department of Mathematics and Computer Science,
University of M\"unster,
Einsteinstraße 62,
48149 M\"unster, Germany.\newline
\url{franziska.jahnke@uni-muenster.de}}
\thanks{
Franziska Jahnke's research is funded by the Deutsche Forschungsgemeinschaft (DFG, German Research Foundation) under Germany's Excellence Strategy EXC 2044/2–390685587 and EXC 2044-390685587, 
Mathematics Münster: Dynamics–Geometry–Structure, as well as by the DFG-ANR grant \emph{AKE Pact}, DFG - project number 545528554.}

\author{Ferr\'eol Lavaud}
\address{Laboratoire de Math\'ematiques,
UFR Sciences et techniques,
16 route de Gray,
25030 Besançon Cedex,
France\newline
\url{ferreol.lavaud@gmail.com}
}
\thanks{This article is based on the MSc thesis of the second author, written under the supervision of the first author at the University of Amsterdam.}

\begin{document}

\begin{abstracts}
\abstractin{english}
    Building on work of the first author and Kartas, we identify the elementary class generated
    by all perfectoid fields of fixed residue characteristic $p$ in the language of rings. 
    
\medskip 

\abstractin{french} En nous appuyant sur les travaux de la première auteure et de Kartas, nous identifions la classe élémentaire générée par les corps perfectoïdes de caractéristique résiduelle $p$ fixée, dans le langage des anneaux.
\end{abstracts}

\maketitle

\section{Introduction}
The aim of this note is to identify the elementary class generated by all perfectoid fields
the language of rings $\Lring = \{0,1,+,\cdot,{-}\}$, that is, the smallest class of fields
which is axiomatized by an $\Lring$-theory and contains all perfectoid fields of a given
residue characteristic $p >0$.
Recall that a field $K$ with a nontrivial valuation $v: K^\times \to \mathbb{R}$ is called perfectoid if
\begin{enumerate}
    \item[(i)] $(K,v)$ is complete, and
    \item[(ii)] $\mathcal{O}_v/(p)$ is semiperfect, where $p = \mathrm{char}(\mathcal{O}_v/
    \mathfrak{m}_v)>0$, and
    \item[(iii)] $v(K^\times)$ is non-discrete.
\end{enumerate}
Here, a ring $R$ of characteristic $p$ is called semiperfect if the Frobenius map $x \mapsto x^p$
is surjective on $R$. See subsection \ref{subsec:not} for details on notation.
Perfectoid fields play a crucial role in perfectoid geometry (cf.~\cite{ScholzeICM}),
and the model-theoretic
properties of perfectoid fields have recently received a considerable amount of attention, e.g.~in work of Kuhlmann--Rzepka \cite{KR21}, Kartas \cite{KK1, Kar25}, Jahnke--Kartas \cite{JK25}, Rideau-Kikuchi--Scanlon--Simon \cite{RKSS25},
Ketelsen \cite{KPhD}, Gitin--Koenigsmann--Stock \cite{GKS26}, and Gambardella--Kartas \cite{gambardella2026transferprincipleskatokuzumakiconjecture}. The model theory of perfectoid fields is also treated in two recent surveys by 
Kuhlmann \cite{Kuh25} and Anscombe \cite{Bour26}.

As evident from the definition of a perfectoid field, one most naturally considers perfectoid fields in 
$\Lval=
\mathcal{L}_\mathrm{ring}\cup \{\mathcal{O}\}$, where $\mathcal{O}$ is a unary predicate interpreted as the valuation ring. 
Clearly, the fact that $v$ is archimedean (i.e., $v(K^\times)\leq \mathbb{R}$) is not elementary
in first-order logic, as it is not preserved under ultrapowers. 
One way to remedy this is by stipulating that the value group should be regular,
which is a first-order property that characterizes those ordered abelian groups $\Loag$- elementary
equivalent to (non-trivial) archimedean ones by the work of Robinson and Zakon \cite{robzak1960},
where $\Loag = \{0,+ \leq\}$.
An ordered abelian group is regular if and only if every quotient by a nontrivial convex subgroup
is divisible \cite{PConrad}.

Moreover,
completeness of $(K,v)$ is not a first-order property either. This can be seen directly from the
L\"owenheim-Skolem theorem \cite[Theorem 2.1.3]{TZ}: Every perfectoid field has a countable elementary substructure, 
and no countable field is complete with respect to a nontrivial
valuation.
A common work-around is to
replace completeness by henselianity, as every valued field with archimedean value group 
which is complete is henselian
\cite[Theorem 1.3.1]{EP05}.

Last but not least, property (iii) is elementary in $\Lval$, and property (ii) is elementary 
for fixed
$p >0$. As the residue characteristic of a perfectoid field is necessarily positive, 
we fix the residue characteristic throughout the paper.

In conclusion, the class
$$\mathcal{C}_p^\mathrm{val}=\{(K,v)\textrm{ henselian}\mid 
v(K^\times) \textrm{ regular and not discrete, } \mathcal{O}_v/(p)\textrm{ semiperfect}\}$$
is $\Lval$-elementary. However, it 
is not generated by the class of perfectoid fields of residue characteristic $p$,
as shown in the next section: it
contains fields not elementarily equivalent to any ultraproducts of perfectoid fields. 
Instead, we focus on the language of rings in section \ref{sec:ring}, for which we obtain
a complete answer.

One characterization given implicitly of the $\Lring$-elementary class
generated by all perfectoid fields of fixed residue characteristic $p$ is given in \cite{JK25}: 
a field is $\Lring$-elementarily equivalent to a perfectoid field of residue
    characteristic $p$ if and only if it is $\Lring$-elementarily equivalent to a field admitting a nontrivial tame valuation with divisible value group and residue characteristic $p$
    (which is an axiomatizable class in $\Lring$). We present this in Proposition \ref{prop:tame}. 

    For fields of characteristic $0$, we present an alternative characterization in Theorem \ref{thm:main}, namely we show that the reduct of $\mathcal{C}_p^\textrm{val}$ to $\Lring$
    (or more precisely its intersection with the class of fields of characteristic $0$) is 
    sufficient:
\begin{thm*} Let $p$ be prime.
The class of fields
\begin{align*}
\mathcal{C}_{(0,p)}^\mathrm{ring} = \{K \mid \,& K \textrm{ admits a henselian valuation $v$ of mixed characteristic } (0,p)\\ &\textrm{ with $\Oo_v/(p)$ semiperfect and $vK$ regular and non-discrete } \}    
\end{align*}
is $\Lring$-elementary and is the $\Lring$-elementary class generated by all
perfectoid fields of characteristic $(0,p)$.
\end{thm*}
In particular, we show that admitting a henselian valuation of mixed characteristic $(0,p)$ with 
$\Oo/(p)$ semiperfect is $\Lring$-axiomatizable.     

\subsection{Notation and nomenclature} \label{subsec:not}
Given a first-order language $\mathcal{L}$, we say that a class of $\mathcal{L}$-structures is \emph{elementary} if it is the model class of an $\mathcal{L}$-theory $T$ 
(where $T$ is not necessarily finite). Similarly, we call a property of $\mathcal{L}$-structures
\emph{axiomatizable} if it can be axiomatized by a (not necessarily finite) $\mathcal{L}$-theory $T$.
If $\mathcal{C}$ is a class of $\mathcal{L}$-structures, the 
\emph{elementary class generated by $\mathcal{C}$}
is the smallest elementary class containing all members of $\mathcal{C}$. 

Throughout, by \emph{valuation}, we mean a Krull valuation and we denote all valuations
additively, following the conventions from \cite{EP05}.
For a valued field $(K,v)$, we write $\mathcal{O}_v$ for the valuation ring,
$\mathfrak{m}_v$ for its maximal ideal,
$Kv$ for its residue field, and $vK$ or $v(K^\times)$ for its value group. We call a valuation
$v$ \emph{discrete} if $vK$ has a minimum positive element, and non-discrete or dense
otherwise. We say that a valuation is \emph{real-valued} if $vK$ is (isomorphic to) a
\emph{subgroup} of the real numbers. For a real-valued field $(K,v)$, we 
write $\widehat{(K,v)}$ or $(\widehat{K}, \widehat{v})$ for its (Cauchy) completion.
 
\subsection*{Acknowledgements}
We wish to thank Sylvy Anscombe, Philip Dittmann, Leo Gitin, and Konstantinos Kartas for their
interest and inspiring discussions. 
Special thanks to Sylvy Anscombe and Philip Dittmann
for their multitude of very helpful comments on a previous version, in particular around 
Lemma \ref{lem-rk1} and \cite{AK16}. 
We also thank the anonymous referee for their detailed reading of the paper and their many 
helpful
comments!
Moreover, our gratitude goes to the organizers of
the DDG seminar (Séminaire Structures algébriques ordonnées), 
a wonderful seminar whose 40th birthday is celebrated with this volume.
Merci beaucoup à vous, Françoise Delon, Max Dickmann et Danielle Gondard, d'avoir organisé cet excellent séminaire pendant plus de quarante ans.
Merci également, Vincent Bagayoko, de t'être joint à eux et pour perpétuer la tradition.

\section{Perfectoid fields in the language of valued fields}
We start by discussing an example of a henselian valued field $(K,v)$ of positive characteristic 
that is in the class $$\mathcal{C}_p^\mathrm{val} = \{ (K,v)\textrm{ henselian}\mid 
vK \textrm{ regular and not discrete, } \mathcal{O}_v/(p)\textrm{ semiperfect}\}$$
as introduced in the introduction, but which is not elementarily equivalent to an ultraproduct 
of perfectoid fields.

Recall that if $(K,v)$ does not admit any proper finite immediate extensions, $(K,v)$ is called \emph{algebraically maximal}.
Note that if $(K,v)$ is a valued field with $vK \leq \mathbb{R}$, 
we have $\Oo_v[\varpi^{-1}] = K$ for any
$\varpi \in \mathfrak{m}_v \setminus \{0\}$, and in particular, any coarsening of $v$
is trivial and thus algebraically maximal. By \cite{JK25}, under certain hypotheses, this property is axiomatizable, using a constant for $\varpi$:

\begin{remark}
    By \cite[Proposition 4.1.4]{JK25}, for henselian valued fields of residue characteristic $p$ with
$\Oo_v/(p)$ semiperfect, the fact that the coarsening $u$ 
corresponding to inverting $\varpi$ (i.e., $\Oo_u=\Oo_v[\varpi^{-1}]$)
is algebraically maximal is an elementary property in
$\Lval$ together with a constant symbol for $\varpi$. More precisely, it is 
axiomatized by the axiom scheme:
\begin{align*} 
 \textrm{  For any finite extension $(K',v')/(K,v)$ with ${v'}K'=vK$,}\\
 \textrm{there is $a \in \Oo_{v'}$ with $0\leq v(\delta(a))\leq  v\varpi$ and $K'=K(a)$.}
 \tag{$\star_\varpi$}
\end{align*}
Here, $\delta(a)$ denotes the different of $a$ (i.e., the derivative of the minimal polynomial
of $a$ evaluated at $a$). The axiom scheme $(\star_\varpi)$ contains
one sentence for each potential degree $[K':K]$. In conclusion, for any perfectoid field $(K,v)$,
we have $$(K,v)\models \forall \varpi \in \mathfrak{m}_v: (\star_\varpi) \textrm{ holds}.$$
\end{remark}

We now use this to show that $\mathcal{C}_p^\mathrm{val}$ is not the $\Lval$-elementary class
generated by all perfectoid fields of residue characteristic $p$:

\begin{ex}
Let $\mathcal{U}$ be a nonprincipal ultrafilter on $\omega$ and consider the ultraproduct 
$$(L_0,w_0)\coloneq (\mathbb{F}_p((t)),v_t)^\mathcal{U}.$$ Let $L$ denote the perfect 
hull of $L_0$, and let $w$ denote the unique extension of $w_0$ to $L$. Then $(L,w)$ is henselian, with value group regular and non-discrete. As $L$ is perfect of characteristic $p$, 
the ring $\Oo_w/(p) = \Oo_w$ is semiperfect.
However, by Example 4.1.5 in \cite{Schaaf}, the property $(\star_\varpi)$ fails in $(L,w)$ (for $p \neq 2$) for all $\varpi \in \mathfrak{m}_w \setminus \{0\}$.
Hence we conclude
 $$(L,w)\models \forall \varpi \in \mathfrak{m}_w: (\star_\varpi) \textrm{ fails}.$$
Thus, $L$ is not $\Lval$-elementarily equivalent to a perfectoid field, nor to an 
ultraproduct of perfectoid fields. 

Since the value group of $v_0L_0$ is regular (it is an
ultrapower of the ordered abelian group $\mathbb{Z}$), so is its $p$-divisible hull $wL$: there is
a one-to-one correspondence of convex subgroups of $v_0L_0$ and $vL$, and so the divisibility
of $v_0L_0/\Delta$ for any proper convex subgroup $\Delta \leq v_0K_0$ 
implies the divisibility of $vL/\Gamma$ for all proper convex subgroups $\Gamma\leq vL$. 
By \cite{PConrad}, this property is equivalent to regularity.
\end{ex}

A corresponding example of a valued field mixed characteristic which is in $\mathcal{C}_p^\mathrm{val}$
but not elementarily equivalent to an ultraproduct of perfectoid fields (due to the failure of 
the same axiom scheme)
was constructed by Leo Gitin. \smallskip

The example above gives rise to the following    
\begin{Q} \label{Q}
    Is the $\Lval$-elementary class generated by perfectoid fields of residue characteristic $p$
    given by the subclass $\mathcal{C}_{p,\star}^\mathrm{val}$ of $\mathcal{C}_p^\mathrm{val}$ such that $(\star_\varpi)$ holds
    for all $\varpi \in \mathfrak{m}_v \setminus \{0\}$?
\end{Q}

As we briefly discuss in the next paragraphs, a potential 
route to answering this question is to
construct an elementarily equivalent real-valued
field given any $(K,v) \in \mathcal{C}_{p,\star}^\mathrm{val}$. Such a construction seems beyond our current techniques, unless $(K,v)$ is tame or of finite rank.

\begin{Def}
    Let $(K,v)$ be a henselian valued field with value group $\Gamma$ and residue field $k$. A finite valued field extension $(K',v')/(K,v)$ is said to be \textit{tame} if the following are satisfied:
\begin{enumerate}
\item  If $p=\text{char}(k)>0$, then $p\nmid [\Gamma':\Gamma]$. 
\item The residue field extension $k'/k$ is separable.
\item The extension $(K',v')/(K,v)$ is \textit{defectless}, i.e., 
$$[K':K]=[\Gamma':\Gamma]\cdot [k':k]$$
\end{enumerate}
We say that $(K,v)$ is \textit{tame} if every finite valued field extension of $(K,v)$ is tame. 
\end{Def}
 Note that tame valued fields form an $\Lval$-elementary class (cf.~\cite[Section 7.1]{Kuh16}) and are in particular henselian. 
All tame fields are algebraically maximal since 
every defectless henselian field is algebraically maximal. Moreover, all coarsenings
of a tame valuation are again tame \cite[Lemma 3.14]{Kuh16} and hence algebraically maximal. By
\cite[Theorem 1.2]{KR21}, for any (non-trivially valued) tame field $(K,v)$ 
of residue
characteristic $p>0$ the residue ring $\mathcal{O}_v/(p)$ is semiperfect. Thus, any tame field
of residue characteristic $p>0$ is in the class
$\mathcal{C}_{p, \star}^\mathrm{val}$.
 Kuhlmann proves Ax--Kochen/Ershov principles for tame fields in \cite{Kuh16}. 
 As a straightforward application of Kuhlmann's results, we show:

\begin{lemma} \label{lem-rk1}
\begin{enumerate}
    \item Any tame field $(K,v)$ of positive characteristic with regular value group 
is $\Lval$-elementarily equivalent to a real-valued field.  
\item Any tame field $(K,v)$ of mixed characteristic 
with divisible value group
is $\Lval$-elementarily equivalent to a real-valued field.    
\end{enumerate}
\end{lemma}
\begin{pen}
\begin{enumerate}
\item If $\mathrm{char}(K)=p>0$, $vK\equiv \Gamma \leq \mathbb{R}$ and $(K,v)$ is tame, 
then \cite[Theorem 1.4]{Kuh16} implies $(K,v) \equiv (Kv((t^\Gamma)),v_t))$.
\item If $(K,v)$ is tame of mixed characteristic $(0,p)$ with $vK \equiv \mathbb{Q}$, then 
consider the subfield $F_0 \subseteq K$ generated by adjoining to $(\mathbb{Q},v_p)$ the lift of a transcendence base of $Kv$
to $K$.
Writing $w_0$ for the Gauss extension of $v_p$ to $F_0$ (see \cite[Corollary 2.2.2]{EP05}), 
we obtain $(F_0, w_0) \subseteq (K,v)$
and $w_0F_0 = \mathbb{Z}$ by construction.
Let now $F$ be the relative algebraic closure of $F_0$ in $K$, and let 
$w$ denote the restriction of $v$ to $F$. Then, we get $Fw=Kv$ and $wF \leq \mathbb{Q}$. 
By \cite[Lemma 3.7]{Kuh16}, we conclude
that $(F,w)$ is an elementary substructure of $(K,v)$ with value group 
$wF=\mathbb{Q}$. \hfill{$\Box$}
\end{enumerate}
\end{pen}

It is not known whether Lemma \ref{lem-rk1}(2) can be strengthened to the analogue of (1) in 
mixed characteristic.%
\footnote{This would follow from \cite{AK16}. However, 
the paper \cite{AK16} contains an error in the proof of Theorem 1.3, as \cite[Theorem 1.3]{AK16}
implies that if $(K,v)$ is a valued tame field and $\Gamma \preceq vK$, 
then there is $(K',v') \preceq (K,v)$ with $v'K'=\Gamma$,
regardless of the characteristic of $(K,v)$. This is not true, as the following example
by Philip Dittmann shows:
Let $p$ be a prime and consider the tame field \label{ex:pd}
$(k,v_t)=(\mathbb{F}_p^\mathrm{alg}((t^{\mathbb{Z}[1/p]})), v_t)$, i.e., the field Laurent series with
coefficients in the $p$-divisible hull of $\mathbb{Z}$. Let now $(K,w)$ be a tame field of characteristic $(0,p)$ with
residue field $k$, value group $\mathbb{Q}$ and algebraically closed algebraic part.
Note that the composition $v$ of $w$ and $v_t$ is again tame. Its value group $vK$ is
the lexicographic sum $\mathbb{Q} \oplus \mathbb{Z}[1/p]$, which is regular and an $\Loag$-elementary
extension of $\Gamma\coloneqq \mathbb{Z}[1/p]$.
However, no elementary substructure $(L,w)$ of $(K,v)$ has value group $\Gamma$, as each such
has algebraically closed algebraic part and hence the value group $wL$ must contain $\mathbb{Q}$.
Presumably, the assumption $v(p) \in \Gamma$ needs to
be added to ensure that \cite[Theorem 1.3]{AK16} holds in mixed characteristic.}

More generally, Ax--Kochen/Ershov type principles are proven for almost tame fields (thus named in
\cite{Bour26}), a class 
including all perfectoid fields in \cite{JK25}. 
Most AK/E type principles reduce the study of a sufficiently well-behaved 
valued field to that of its value group and residue field.
Note that perfectoid fields may
admit immediate algebraic extensions: for example, 
adjoining an Artin-Schreier root of $1/t$ gives an
immediate extension of the perfectoid field $(\widehat{\mathbb{F}_p((t))^\mathrm{perf}}, v_t)$, i.e., the Cauchy completion of the perfect hull of $\mathbb{F}_p((t))$.
Thus, there is little hope of reducing the theory of a perfectoid field down to the theories
of its residue field and value group. Instead, the residue field is replaced by an
infinitesimal thickening, namely the residue ring $\mathcal{O}_v/(\varpi)$ for some $\varpi \in \mathfrak{m}_v \setminus \{0\}$. The corresponding AK/E type result is then the following:

\begin{thm}[{\cite[Theorem 5.1.4]{JK25}}] \label{thm:perfectoidAKE}
Let $(K,v)\subseteq (K',v')$ be two henselian valued fields of residue characteristic $p>0$ such that $\Oo_v/(p)$ and $\Oo_{v'}/(p)$ are semiperfect. Suppose there is $\varpi \in \mathfrak{m}_v$ such that both $\Oo_v[\varpi^{-1}]$ and $\Oo_{v'}[\varpi^{-1}]$ are algebraically maximal. Then the following are equivalent: 
\begin{enumerate}[label=(\roman*)]
\item $(K,v)\preceq (K',v')$ in $\Lval$.
\item \label{condition2} $\mathcal{O}_v/(\varpi )\preceq \mathcal{O}_{v'}/(\varpi)$ in $\Lring$ and $vK\preceq v'K'$ in $\mathcal{L}_{\mathrm{oag}}$.
\end{enumerate}
Moreover, if $vK$ and $v'K'$ are regularly dense, then the value group condition in \ref{condition2} can be omitted. 
\end{thm}
Note that, as discussed at the beginning of the section,
Theorem \ref{thm:perfectoidAKE} applies in particular to any real-valued field in $\mathcal{C}_p^\mathrm{val}$.
As a minor application, we 
see immediately that any real-valued field $(K,v) \in \mathcal{C}_p^\mathrm{val}$ is
$\Lval$-elementarily equivalent to a perfectoid field:
\begin{cor} \label{cor:complete}
\begin{enumerate}
    \item
Let $(K,v) \in \mathcal{C}_p^\mathrm{val}$ such that $vK \leq \mathbb{R}$ and let $(K',v')$ be its completion. Then $(K',v')$ is perfectoid and the embedding $(K,v) \subseteq (K',v')$ is elementary.
\item
Conversely, if $(K',v')$ is perfectoid and $(K_0,v_0)\subseteq (K',v')$ is any subfield such that 
the extension is immediate, i.e., we have $K_0v_0 = Kv$ and $v_0K_0=Kv$. Then,
the extension $(K,v) \coloneqq (K_0^\mathrm{perf},v_0)^h \subseteq (K',v')$ is elementary.
\end{enumerate}
\end{cor}
\begin{proof} 
Note that if $vK \leq \mathbb{R}$, and $(K',v')$ is the completion of $(K,v)$, and so $(K,v)$ is
dense in $(K',v')$. In particular, for any $\gamma \in vK$ and $a \in K'$ there is $b\in K$ with
$v'(a-b) > \gamma$. 
Thus, for
any $\varpi \in\mathfrak{m}_v\setminus \{0\}$ we get
$\Oo_{v'}/(\varpi) = \Oo_v/(\varpi)$: we get a canonical 
embedding of $\Oo_v/(\varpi)$ into $\Oo_{v'}/(\varpi)$ since 
$\varpi \Oo_{v'}\cap \Oo_v = \varpi \Oo_v$, and this map is surjective as given 
$a \in \Oo_{v'}$, any $b \in \Oo_v$ with
$v'(a-b) > v(\varpi)$ is equivalent to $a$ modulo $\varpi \Oo_{v'}$.
Hence, both (i) and (ii) follow immediately from Theorem \ref{thm:perfectoidAKE}.
For (ii), note that as $(K',v')$ is an immediate complete extension of $(K,v)$ (which in turn embeds to $(K',v')$ by the universal property of the henselization), 
it is in fact its completion.
\end{proof}

Thus, finding real-valued models of the $\Lval$-theory of
any $(K,v) \in \mathcal{C}^\mathrm{val}_{p,\star}$ would answer Question \ref{Q}.

\section{Perfectoid fields in the language of rings}
\label{sec:ring}
The aim of this section is to characterize the $\Lring$-theory generated by the
class of perfectoid fields of residue
characteristic $p$. 
Throughout this section, we use the following to show that certain classes
of fields are $\Lring$-elementary:
\begin{fact}[{\cite[Theorem 4.1.12]{changkeisler}} and {\cite[Exercise 4.1.17]{changkeisler}}] Let $\mathcal{L}\subseteq \mathcal{L}'$ \label{fact}
be first-order languages.
\begin{enumerate}
    \item
    A class of $\mathcal{L}$-structures is elementary if and only if it is closed under
    elementary equivalence and ultraproducts. 
    \item Let $\mathcal{C}'$ be an elementary class of $\mathcal{L}'$-structures
    and let $\mathcal{C}$ be class of all reducts of elements of $\mathcal{C}'$ to $\mathcal{L}$.
    Then $\mathcal{C}$ is closed under ultraproducts.
\end{enumerate}
\end{fact}

Note that a non-algebraically closed field $K$ can admit at most 
one valuation ring $\Oo_v$ such that $(K,v)$ is
perfectoid: Any two non-trivial henselian valuation rings on a field have a non-trivial common coarsening, unless the field is separably closed (\cite[Theorem 4.4.1]{EP05}). However, since a 
perfectoid valuation on a field $K$ is by definition real-valued, it never admits 
any non-trivial coarsenings. As any perfectoid field $(K,v)$ is perfect,
$K$ can only admit a different perfectoid valuation rings if it is algebraically closed.\footnote{On the other hand, if $K$ is algebraically closed, this occurs naturally: writing $(\mathbb{C}_p,v_p)$ for the completion
of the algebraic closure of $(\mathbb{Q}_p,v_p)$, for $p\neq l$ we have an isomorphism 
$\varphi: \mathbb{C}_p \cong \mathbb{C}_l$ as fields - these are algebraically closed fields of characteristic zero which have the 
same cardinality. Now $\mathbb{C}_p$ admits the perfectoid valuation rings $\Oo_{v_p}$ and $\varphi(\Oo_{v_l})$ which differ as their residue characteristics are different.}

We first prove the following characterization, which is implicit in \cite{JK25}:
\begin{prop} \label{prop:tame}
Let $L$ be any field. 
\begin{enumerate}
    \item If $L$ admits a
    nontrivial tame valuation with divisible value group and residue characteristic $p>0$, then
    $L$ is $\Lring$-elementarily equivalent to a field admitting a non-trivial 
    perfectoid valuation 
    of residue characteristic
    $p>0$.
    \item Conversely, if $L$ admits a non-trivial perfectoid valuation 
    of residue characteristic $p>0$, it is $\Lring$-elementarily equivalent to a field admitting a non-trivial tame valuation with divisible value group and residue characteristic $p>0$.
\end{enumerate}
\end{prop}
\begin{pen}
\begin{enumerate}
    \item 
    Assume that $L$ admits a nontrivial tame henselian valuation $v$
    with divisible value group and residue characteristic $p>0$. By Lemma \ref{lem-rk1}, there
    is some $(M,w)\equiv (L,v)$ such that $wM \leq \mathbb{R}$. Since completions are
    immediate extensions and the completion of a rank-$1$ tame field is again tame, by \cite{Kuh16}), we obtain $(M',w')\coloneqq \widehat{(M,w)}\succeq (M,w) \equiv (L,v)$.
    Since $\Oo_{w'}/(p)$ is semiperfect by 
    \cite[Theorem 1.2]{KR21}),  $(M',w')$ is perfectoid. Thus, $(L,w)$ is $\Lval$-elementarily
    equivalent to the perfectoid field $(M',w')$, and so
    $L$ is in particular $\Lring$-elementarily equivalent to $M'$, which admits a non-trivial
    perfectoid valuation of residue characteristic $p>0$.
    \item
    Let $(L,v)$ be perfectoid, and let $(L,v)^\mathcal{U}$ be any nonprincipal
    ultrapower. Choose
    any $\varpi \in \mathfrak{m}_v\setminus \{0\}$, and let $w$ denote the coarsest coarsening of
    $v^\mathcal{U}$ with $w(\varpi)>0$, i.e., $\mathcal{O}_w = \mathcal{O}_{v^\mathcal{U}}[1/\varpi]$. Note that, by $\aleph_1$-saturation of $L^\mathcal{U}$, 
    $w$ is necessarily non-trivial.
    By \cite[Main Theorem (I)]{JK25}, all finite extensions of $(L^\mathcal{U},w)$ are unramified, which precisely means that $w$ a 
    tame valuation
    with divisible value group of residue characteristic $p>0$. 
    This finishes the proof, as by construction 
    $L \equiv L^\mathcal{U}$ as $\Lring$-structures.
    \hfill$\Box$
    \end{enumerate}
\end{pen}

We now apply the above to give our first characterization of 
the $\Lring$-elementary class generated by all fields admitting a perfectoid valuation
of residue characteristic $p>0$.
In \cite{AJ15}, fields $\Lring$-elementarily equivalent to fields admitting nontrivial tame valuations with divisible value group were called \emph{t-henselian of divisible tame type}.

\begin{remark}
    The class of fields
     \begin{align*}
    \mathcal{C}_p^\mathrm{ring} = \{ K \mid \;& K \equiv_{\Lring} L\textrm{ for some $L$ admitting a non-trivial
     valuation
     $w$ such that}\\ &\textrm{$(L,w)$ tame, $wL$ divisible, }\mathrm{char}(Lw)=p \}
      \end{align*}
    is $\Lring$-elementary: by definition it is closed under elementary equivalence. As admitting a nontrivial
    tame henselian valuation with divisible value group and residue characteristic $p>0$ is
    $\Lval$-axiomatizable, $\mathcal{C}_p^\mathrm{ring}$
    is also closed under ultraproducts. Thus, $\mathcal{C}_p^\mathrm{ring}$ is indeed
    an $\Lring$-elementary class (cf.~Fact \ref{fact}).
    By Proposition \ref{prop:tame}, $\mathcal{C}_p^\mathrm{ring}$ is indeed the $\Lring$-elementary class generated
    by all perfectoid fields of residue characteristic $p>0$.
\end{remark}

As a straightforward consequence of Proposition \ref{prop:tame}, we obtain:
\begin{cor} \label{UltraProdPerfElemRing}
    Any ultraproduct of fields admitting non-trivial perfectoid valuations 
    of fixed residue characteristic $p>0$
    is $\Lring$-elementarily equivalent to a perfectoid field.
\end{cor}
Whether this generalizes to $\Lval$ remains open:
\begin{Q}
 Is any ultraproduct of perfectoid valued fields of fixed residue characteristic $p>0$ 
 already $\Lval$-elementarily equivalent to a perfectoid field $(K,v)$?
\end{Q}

We now show that for fields of characteristic $0$, the $\Lring$-reduct of the class 
$\mathcal{C}_p^\mathrm{val}$ is $\Lring$-elementary and is the indeed the $\Lring$-elementary class generated
by all perfectoid fields of mixed characteristic $(0,p)$. We begin with the following well-known observations:

\begin{lemma} \label{lem-mixed}
    \begin{enumerate}
    \item Admitting a henselian valuation of mixed characteristic $(0,p)$ is $\Lring$-axiomatizable.
        \item The valuation ring of any unramified henselian valuation of mixed characteristic $(0,p)$ (i.e., $v(p)$ minimum positive) is $\Lring$-definable without parameters (by a formula depending only on $p$).
        Thus, admitting an unramified henselian valuation of mixed characteristic $(0,p)$
     is $\Lring$-axiomatizable.
     \end{enumerate}
\end{lemma}
\begin{proof}
  \begin{enumerate}
      \item 
    Admitting a henselian valuation of mixed characteristic $(0,p)$ is preserved
    under $\Lring$-elementary equivalence (this is also shown in \cite[Proposition 2.1]{AJ15}): assume $K \equiv L$ and $K$ admits a henselian valuation $v$ of mixed characteristic. By the 
    Keisler-Shelah Theorem (\cite[Theorem 6.1.15]{changkeisler}), there is an ultrafilter $\mathcal{U}$ with $K^\mathcal{U} \cong
    L^\mathcal{U}$ as fields.
    Note that $K^\mathcal{U}$
    also admits a henselian valuation of mixed characteristic (namely $v^\mathcal{U}$). Restricting $v^\mathcal{U}$ to $L$ gives a henselian valuation 
    (since $L$ is relatively algebraically closed in the ultrapower), which is of mixed characteristic (as $p \in \mathfrak{m}_{v^\mathcal{U}}$). Thus, admitting a henselian
    valuation of mixed characteristic is preserved under $\Lring$-elementary equivalence.
    Since the class of henselian valued fields of mixed characteristic $(0,p)$ is
    $\mathcal{L}_\mathrm{val}$-elementary, its reduct to
    $\Lring$ is closed under ultraproducts by Fact \ref{fact}(2). Thus,
    admitting a henselian valuation of mixed characteristic $(0,p)$ is 
    $\Lring$-axiomatizable by Fact \ref{fact}(1).
\item
    Any mixed characteristic henselian valuation $(K,v)$ with $v(p)$ minimum positive is
    $\Lring$-definable by a classical result of Robinson:
    For $p\neq 2$, the formula $\varphi_p(x): (\exists y) [1 + px^2 = y^2]$
    defines $\mathcal{O}_v$, for $p=2$, $\Oo_v$ is definable by
    $\varphi_2(x): (\exists y) [1 + 2x^3 = y^3]$.
    Thus, if $K$ admits a henselian valuation $v$ of mixed 
    characteristic $(0,p)$ with $v(p)$ minimum positive, then so does any $L\equiv K$:
    the set defined by $\varphi_p(x)$ in $L$ is the valuation ring of a valuation $w$
    with $(K,v) \equiv (L,w)$.
    Moreover, admitting a henselian valuation of mixed characteristic $(0,p)$ which is unramified
    is preserved under ultraproducts by Fact \ref{fact}(2). We conclude once again by Fact \ref{fact}(1).
  \end{enumerate}
\end{proof}

Note that if $v$ is a valuation of mixed characteristic on $K$ and $vK$ has a minimum positive element, $\Oo_v/(p)$ is semiperfect implies that $v(p)$ is minimum positive in $vK$, since
any $\gamma\in vK$ with $0\leq \gamma <v(p)$ is $p$-divisible by semiperfectness. 
In this case
it may happen that $K$ admits two henselian valuations of mixed characteristic $v$ and $w$ 
such that $v$ is unramified and $\Oo_w \subseteq \Oo_v$ holds with $\Oo_v/(p)$ semiperfect
but $\Oo_w/(p)$ not semiperfect (see \cite[Remark 4.16]{GKS26} for an explicit example).
This is however the only obstruction:
\begin{prop} \label{prop:o/p}
    Let $K$ be a field and $p$ a prime, and assume that $K$ admits no unramified henselian
    valuation of mixed characteristic $(0,p)$. Then $\mathcal{O}_v/(p)$ is semiperfect for
    \emph{some} henselian valuation $v$ on $K$ with $\mathrm{char}(Kv)=p$ if and only if 
    $\mathcal{O}_v/(p)$ is semiperfect for
    \emph{all} henselian valuations $v$ on $K$ with $\mathrm{char}(Kv)=p$.
\end{prop}
\begin{proof}
The statement is clear if $\mathrm{char}(K)=p$, in which case the semiperfectness of $\Oo_v/(p)
=\Oo_v$
is equivalent to the perfectness of $K$. Thus, we now assume $\mathrm{char}(K)=0.$

If $K$ is algebraically closed, we have $\Oo_v/(p)$ is semiperfect for any valuation $v$ on $K$.
If $K$ is not algebraically closed, all henselian valuations on $K$ 
induce the same topology \cite[Theorem 4.4.1]{EP05}.
Let $v$ and $w$ be henselian valuations of mixed characteristic $(0,p)$ on $K$, and assume 
$\Oo_v/(p)$ is semiperfect. As $v$ and $w$ induce the same topology, they have a finest common coarsening $u$ which is nontrivial. We make a case distinction depending on the residue characteristic of $u$.

If $\mathrm{char}(Ku)=0$, then $u$ is a proper coarsening of both $v$ and $w$ and $Ku$ is algebraically closed (by another application of \cite[Theorem 4.4.1]{EP05}. Thus, writing 
$\bar{w}$
for the valuation induced by $w$ on $Ku$, the ring 
$\Oo_{\bar{w}}/(p)$ is semiperfect by algebraic closedness of $Ku$.
Considered as ideals of $\Oo_w$, we get $\mathfrak{m}_u\subseteq (p)$, and so the residue map corresponding to $u$ induces an isomorphism $\Oo_{\bar{w}}/(p) \cong \Oo_{{w}}/(p)$.
Hence, $\Oo_w/(p)$ is also semiperfect.

If $\mathrm{char}(Ku)=p$, we may assume that
both $(K,v)$ and $(K,w)$ are $\aleph_1$-saturated 
as, by Lemma \ref{lem-mixed}(2), the assumptions still hold when passing to an ultrapower of $(K,v,w)$.
We write $v^+$ (respectively $w^+$) to denote the
coarsest coarsening of residue characteristic $p$ of $v$ (respectively $w$).
Since both $v^+$ and $w^+$ are
(not necessarily proper) coarsenings of $u$ and 
coarsenings are
linearly ordered by inclusion, we obtain $v^+=w^+$. 
Since by assumption $K$ admits no unramified henselian valuation of mixed characteristic $(0,p)$,
$w^+K = v^+K$ has no minimum positive element.
This allows us to apply \cite[Lemma 7.2.19]{JK25} with $\varpi = p$, and we obtain
\begin{align*}
\Oo_v/(p) \textrm{ is semiperfect} \overset{\textrm{{\cite[Lemma 7.2.19]{JK25}}}}{\Longleftrightarrow} &\;\Oo_{v^+}/(p) \textrm{ is semiperfect}\\ \overset{\Oo_{v^+} = \Oo_{w^+}}{\Longleftrightarrow}\phantom{\Longrightarrow} &\;\Oo_{w^+}/(p) \textrm{ is semiperfect} \\\overset{\textrm{{\cite[Lemma 7.2.19]{JK25}}}}{\Longleftrightarrow} &\;\Oo_w/(p) \textrm{ is semiperfect,}
\end{align*}
as desired. 
\end{proof}

This gives us the following:
\begin{prop} \label{HensVal0pSemiperfElem}
    Admitting a henselian valuation $v$ 
    of characteristic $(0,p)$ with $\Oo_v/(p)$ semiperfect is $\Lring$-axiomatizable.
\end{prop}
\begin{proof} Let $K$ be a field admitting a henselian valuation $v$ 
of mixed characteristic $(0,p)$
with $\Oo_v$ semiperfect. We first show that any $L \equiv_{\Lring} K$ admits a mixed characteristic $(0,p)$ henselian valuation with $\Oo_v/(p)$ semiperfect.

If $K$ admits an unramified henselian valuation $w$ 
of characteristic $(0,p)$, then $w$ must be a coarsening of $v$: If $w$ and $v$ were not comparable, their finest common coarsening $u$ would have separably closed residue field, but then $w$
would induce a valuation on the separably closed field $Ku$ with non-divisible value group.
As $w$ has no proper coarsenings of residue characteristic $p$, we obtain $\Oo_v \subseteq \Oo_w$.
Then, since $(p) = \mathfrak{m}_w \subseteq \mathfrak{m}_v$, the semiperfectness of $\Oo_v/(p)$ 
implies that $\Oo_w/(p) = Kw = \mathrm{Frac}(\Oo_v/(p))$ is perfect. Thus, the parameter-free 
$\Lring$-definability
of $\Oo_w$ (see Lemma \ref{lem-mixed}) 
ensures that any $L \equiv_{\Lring} K$ admits
an unramified henselian valuation $u$ with perfect residue field $Lu=\mathcal{O}_u/(p)$.

We now assume that $K$ admits no unramified henselian valuation of characteristic $(0,p)$.
Let $L \equiv_{\Lring} K$, then $L$ admits a henselian valuation $w$ of characteristic $(0,p)$
by Lemma \ref{lem-mixed}(1). Let $\mathcal{U}$ be an ultrafilter such that
$K^\mathcal{U} \cong L^\mathcal{U}$ (which exists by the Keisler-Shelah Theorem \cite[Theorem 6.1.15]{changkeisler}), then both
$v^\mathcal{U}$ and $w^\mathcal{U}$ are henselian valuations of mixed characteristic on $K^\mathcal{U}.$ By Lemma \ref{lem-mixed}(2), $K^\mathcal{U}$ admits no unramified henselian valuation of characteristic $(0,p)$, so the semiperfectness of $\Oo_{v^\mathcal{U}}/(p)$
implies that of $\Oo_{w^\mathcal{U}}/(p)$ by Proposition \ref{prop:o/p}. We conclude that 
$\Oo_w/(p)$ is semiperfect.

By Fact \ref{fact}(1), what is left to show is that admitting a henselian valuation
mixed characteristic $(0,p)$ with $\Oo_v/(p)$ semiperfect is preserved under taking ultraproducts.
As these properties are all $\Lval$-axiomatizable, this follows once again from Fact \ref{fact}(2).
\end{proof}

We are now in a position to give the desired characterization of the $\Lring$-elementary
class generated by all perfectoid fields of characteristic $(0,p)$:
\begin{thm} Let $p$ be prime. \label{thm:main}
The class of fields
\begin{align*}
\mathcal{C}_{(0,p)}^\mathrm{ring} = \{K \mid \,& K \textrm{ admits a henselian valuation of mixed characteristic } (0,p)\\ &\textrm{ with $\Oo_v/(p)$ semiperfect and $vK$ regular and non-discrete} \}    
\end{align*}
is $\Lring$-elementary and is the $\Lring$-elementary class generated by all
perfectoid fields of characteristic $(0,p)$.
\end{thm}
\begin{proof} Clearly, every field of characteristic $0$ which is perfectoid with respect to a valuation of residue characteristic $p$ is in $\mathcal{C}_{(0,p)}^\mathrm{ring}$.
We first show that the class $\mathcal{C}_{(0,p)}^\mathrm{ring}$ contains only 
fields $\Lring$-elemen\-tarily
equivalent to perfectoid fields of characteristic $(0,p)$.
Let $K \in \mathcal{C}_{(0,p)}^{\textnormal{ring}}$, and let $v$ be the 
henselian valuation on $K$ with $vK$ regular and non-discrete and $\Oo_v/(p)$
semiperfect. 

We first show the following:
\begin{claim}
    $K$ admits no unramified henselian valuation of characteristic $(0,p)$.
\end{claim}
\begin{claimproof}
Assume that $\nu$ was an unramified henselian valuation on $K$ of characteristic $(0,p)$. 
As in the proof of Proposition \ref{prop:o/p}, we argue that $\nu$ would then be a 
coarsening of $v$: If $\nu$ and $v$ were not comparable, their finest common coarsening $u$ would have separably closed residue field, but then $\nu$
would induce a valuation on the separably closed field $Ku$ with non-divisible value group.
As $\nu$ has no proper coarsenings of residue characteristic $p$, we obtain $\Oo_v \subseteq \Oo_\nu$.
Since $v$ itself is not discrete, $\nu$ must even be a proper coarsening of $v$. But, as $vK$
is regular, all proper coarsenings of $v$ have divisible value group, a contradiction.
\end{claimproof}

Let now $(L,w)=(K,v)^\mathcal{U}$ be a nonprincipal ultrapower, and let $w^+$
denote the coarsest coarsening of $w$ with $\mathrm{char}(Lw^+)=p$, and $w^0$ denote the
finest coarsening of $w$ (and hence of $w^+$) of residue characteristic $0$, i.e., 
$\Oo_{w^0} = \Oo_{w}[1/p] = \Oo_{w^+}[1/p]$. 
We now show the following:
\begin{claim} $(L,w^+)$ is tame with divisible value group.
\end{claim}
\begin{claimproof}
We prove first that both $w^0$ and the valuation
$\bar{w}^+$ induced by $w^+$ on $Lw^0$ are tame with divisible value group.
This is clear for $w^0$: since $w^0$ is a henselian valuation of equicharacteristic $0$, it is tame,
and as it is a proper coarsening of $w$, the regularity of $wL$ implies that
$w^0L$ is divisible.

We now argue that $\bar{w}^+$ is also tame with divisible value group.
By Claim 1 and Lemma \ref{lem-mixed}(2), $L$ admits no unramified henselian valuation of characteristic $(0,p)$.
In particular, $w^+$ is not unramified, and hence the rank-$1$ valuation $\bar{w}^+$ is 
maximal and has value group isomorphic to $\mathbb{R}$ (see 
\cite[Theorem 1.13]{AK16}, this is were
$\aleph_1$-saturation of $(L,w)$ is needed). As the residue field of $\bar{w}^+$ (which coincides
with the residue field of $w^+$) is $\mathrm{Frac}((\Oo_w/(p))_\mathrm{red})$, it is perfect.
By \cite[Theorem 3.2]{Kuh16}, we conclude that 
$\bar{w}^+$ is tame with divisible value group. 

As compositions of tame valuations with divisible value group are again tame with divisible value group (this is a well-known fact, see 
\cite[Lemma 2.9]{AJ24}
for an explicit proof that defectlessness composes well), we conclude that $(L,w^+)$ is tame with 
divisible value group.
\end{claimproof}

Now, by Proposition \ref{prop:tame},
$L$ (and hence also $K$) 
is $\Lring$-equivalent to a perfectoid field of residue characteristic $p$.

\smallskip

We show that the class 
$\mathcal{C}^\mathrm{ring}_{(0,p)}$ is $\Lring$-elementary by arguing that it is closed under ultraproducts and elementary equivalence (cf.~Fact \ref{fact}(1)). 
Another way one could argue is that $\mathcal{C}^\mathrm{ring}_{(0,p)}$ coincides
with the intersection of the $\Lring$-elementary class $\mathcal{C}_p^\mathrm{ring}$ defined above and the
$\Lring$-elementary class of fields of characteristic $0$, i.e., 
via Proposition \ref{prop:tame}. We give a self-contained proof without invoking
\cite[Main Theorem]{JK25}.

That $\mathcal{C}_{(0,p)}^{\textnormal{ring}}$ is closed under ultraproducts follows once
again from Fact \ref{fact}(2):
being a henselian valuation 
of mixed characteristic with $\Oo_v/(p)$ semiperfect and
$vK$ regular and non-discrete is $\Lval$-axiomatizable.

Finally, we verify that $\mathcal{C}_{(0,p)}^\mathrm{ring}$ is closed under $\Lring$-elementary equivalence. Take $K \in \mathcal{C}_{(0,p)}^\mathrm{ring}$, let $v$ be the valuation on $K$ witnessing that $K \in \mathcal{C}_{(0,p)}^\mathrm{ring}$, and let $L \equiv_{\Lring} K$.
Then, by Lemma \ref{lem-mixed}(1) and Proposition \ref{HensVal0pSemiperfElem}, $L$ admits a henselian valuation $w$ of mixed characteristic $(0,p)$ with $\Oo_w/(p)$ semiperfect.
As $K$ admits no unramified henselian valuation of characteristic $(0,p)$ 
by Claim 1, neither does $L$ by Lemma \ref{lem-mixed}(1).
In particular, $wL$ is non-discrete.
Without loss of generality, by Proposition \ref{prop:o/p}, $w$ admits no proper 
coarsenings of residue characteristic $p$. 

We claim that $wL$ is regular. The proof structure 
is similar to that of Proposition \ref{prop:o/p}.
Applying the Keisler-Shelah Theorem (\cite[Theorem 6.1.15]{changkeisler}) once again, let 
$M \cong K^\mathcal{U} \cong L^\mathcal{U}$
be isomorphic to nonprincipal ultrapowers of $K$ and $L$, and denote by $v^\mathcal{U}$ and
$w^\mathcal{U}$ the valuations on $M$ corresponding to the ultrapowers of $v$ and $w$.
We show that $w^\mathcal{U}M$ is regular.
If $M$ is algebraically closed, $w^\mathcal{U}$ is divisible and in particular regular.
Otherwise, $w^\mathcal{U}$ and $v^\mathcal{U}$ have a finest common coarsening $u$ 
\cite[Theorem 4.4.1]{EP05}. If $u$ 
is a proper coarsening of both, $Mu$ is algebraically closed. As $u$ is a proper coarsening of $v^\mathcal{U}$ which has regular value group, the value group $Mu$ is divisible. As $Mu$ is algebraically closed, the valuation induced by $w^\mathcal{U}$ on $Mu$ also has divisible value group. Thus, the value group of $w^\mathcal{U}$ is again divisible hence regular. 
Hence, we may assume that $w^\mathcal{U}$ and $v^\mathcal{U}$ are comparable. If $w^\mathcal{U}$ is a proper
coarsening of $v^\mathcal{U}$, its value group is once again divisible. Thus, we are left with
the case that $v^\mathcal{U}$ is a 
coarsening of $w^\mathcal{U}$. Let $\langle w^\mathcal{U}(p) \rangle$ be the convex subgroup of $w^\mathcal{U}M$ 
generated by $w^\mathcal{U}(p)$. Then, as the coarsening $\nu$ of $w^\mathcal{U}$ defined
by the composition 
$\nu: M^\times \overset{w^\mathcal{U}}{\to} w^\mathcal{U}M \to w^\mathcal{U}M/ \langle w^\mathcal{U}(p) \rangle$
has residue characteristic $0$, it is in particular a proper coarsening of $v^\mathcal{U}$
and hence has divisible value group. 
This is $\Lval$-axiomatizable (using the $\emptyset$-definable constant $w^\mathcal{U}(p)$): the quotient $w^\mathcal{U}M/ \langle w^\mathcal{U}(p) \rangle$ is $n$-divisible for $n > 0$ if and only if 
$$w^\mathcal{U}M \models \underbrace{ (\forall \gamma) (\exists \delta)(\exists\varepsilon)\, [\gamma = n\cdot \delta + \varepsilon\, \wedge\, -n\cdot w^\mathcal{U}(p) < \varepsilon < w^\mathcal{U}(p)}_{\varphi_n}].$$
Indeed, it is clear that if $w^\mathcal{U} \models \varphi_n$ for each $n \geq 2$, then
$w^\mathcal{U}M/ \langle w^\mathcal{U}(p) \rangle$ is $n$-divisible. Conversely, assume 
that $w^\mathcal{U}M/ \langle w^\mathcal{U}(p) \rangle$ is $n$-divisible, i.e.,
for every $\gamma$ we find $\delta$ with 
$\gamma - n\delta = \varepsilon \in \langle w^\mathcal{U}(p) \rangle$. We want to show that
we can choose $\varepsilon \in [-w^\mathcal{U}(p), w^\mathcal{U}(p)]$. By replacing $\gamma$
with $-\gamma$ if necessary, we may assume $\varepsilon \geq 0$.
Let $m \in \mathbb{N}$ be maximal such that $n\cdot (m+1) \cdot w^\mathcal{U}(p) > \varepsilon \geq 
n\cdot m \cdot w^\mathcal{U}(p)$. Then 
$$\varepsilon' = \varepsilon - \underbrace{n\cdot m \cdot w^\mathcal{U}(p)}_{\in \,n\cdot w^\mathcal{U}M}< n\cdot w^\mathcal{U}(p), $$
so $\varepsilon'$ witnesses that $\varphi_n$ holds.

Since the embedding $wM \leq w^\mathcal{U}M$ is an elementary embedding of ordered abelian groups, 
we conclude that $ wM/ \langle w(p) \rangle$ is also divisible.  As $w$ has no coarsenings of residue characteristic $p$, any proper convex subgroup of $wL$ contains $w(p)$. Thus,
we have shown that all proper quotients of $wL$ by convex subgroups are divisible, 
hence $wL$ is regular.
\end{proof}

\printbibliography

\end{document}